\documentclass[12pt]{amsart}
\usepackage{graphicx}
\usepackage{a4wide}
\usepackage{amssymb}
\usepackage{amsthm}
\usepackage{amsmath}
\usepackage{amscd} \usepackage{verbatim}
\usepackage[all]{xy}

\textheight8.2in \textwidth6.25in \numberwithin{equation}{section}
\theoremstyle{plain}
\newtheorem{theorem}{Theorem}[section]
\newtheorem{corollary}[theorem]{Corollary}
\newtheorem{lemma}[theorem]{Lemma}

\theoremstyle{definition}

\newtheorem{example}[theorem]{Example}

\theoremstyle{remark}
\newtheorem*{remark}{Remark}

\newcommand{\Q}{\mathbb{Q}}
\newcommand{\Z}{\mathbb{Z}}

\newcommand{\C}{\mathbb{C}}

\renewcommand{\H}{\mathbb{H}}

\newcommand{\leg}[2]{\left( \frac{#1}{#2} \right)}

\newcommand{\ord}{\text {\rm ord}}

\begin{document}

\title[hook lengths and $3$-cores]
{hook lengths and $3$-cores}

\author{Guo-Niu Han and Ken Ono}

\address{I.R.M.A., UMR 7501, Universit\'e Louis Pasteur
et CNRS, 7 rue Ren\'e-Descartes, F-67084 Strasbourg, France}
\email{guoniu@math.u-strasbg.fr}

\address{Department of Mathematics, University of Wisconsin,
Madison, Wisconsin 53706} \email{ono@math.wisc.edu}
\thanks{The second author thanks
the support of the NSF, and he thanks the Manasse family.}

\begin{abstract} Recently, the first  author generalized a formula
of Nekrasov and Okounkov which gives a combinatorial formula, in
terms of hook lengths of partitions, for the coefficients of certain
power series. In the course of this investigation, he conjectured
that $a(n)=0$ if and only if $b(n)=0$, where integers $a(n)$ and
$b(n)$ are defined by
\begin{displaymath}
\begin{split}
\sum_{n=0}^{\infty}a(n)x^n&:=\prod_{n=1}^{\infty}(1-x^n)^8,\\
\sum_{n=0}^{\infty}b(n)x^n&:=\prod_{n=1}^{\infty}\frac{(1-x^{3n})^3}{1-x^n}.
\end{split}
\end{displaymath}
The numbers $a(n)$ are given in terms of hook lengths of partitions,
while $b(n)$ equals the number of $3$-core partitions of $n$. Here
we prove this conjecture.
\end{abstract}

\maketitle

\section{Introduction and statement of results}

In their work on random partitions and Seiberg-Witten theory,
Nekrasov and Okounkov \cite{NO} proved the following striking
formula:
\begin{equation}\label{NOformula}
F_z(x):=\sum_{\lambda} x^{|\lambda|}\prod_{h\in
\mathcal{H}(\lambda)}
\left(1-\frac{z}{h^2}\right)=\prod_{n=1}^{\infty}(1-x^n)^{z-1}.
\end{equation}
Here the sum is over integer partitions $\lambda$, $|\lambda|$
denotes the integer partitioned by $\lambda$, and
$\mathcal{H}(\lambda)$ denotes the multiset of classical hooklengths
associated to a partition $\lambda$. In a recent preprint, the first author
\cite{Han1} has obtained an extension of (\ref{NOformula}),
one which has a specialization which gives the classical generating
function
\begin{equation}\label{tcoregenfcn}
C_t(x):=\sum_{n=0}^{\infty}c_t(n)x^n=\prod_{n=1}^{\infty}\frac{(1-x^{tn})^t}{1-x^n}
\end{equation}
for the number of $t$-core partitions of $n$. Recall that a
partition is a {\it $t$-core} if none of its hook lengths are
multiples of $t$.

In the course of his work, the first author \cite{Han2} formulated a number of
conjectures concerning hook lengths of partitions. One of
these conjectures is related to classical identities of
Jacobi. For positive integers $t$, he compared the functions
$F_{t^2}(x)$ and $C_t(x)$. If $t=1$, we obviously have that
$$
F_1(x)=C_1(x)=1.
$$
For $t=2$, by two famous identities of Jacobi, we have
\begin{displaymath}
\begin{split}
F_4(x)&=\prod_{n=1}^{\infty}(1-x^n)^3=\sum_{k=0}^{\infty}(-1)^k(2k+1)x^{(k^2+k)/2},\\
C_2(x)&=\prod_{n=1}^{\infty}\frac{(1-x^{2n})^2}{1-x^n}=\sum_{k=0}^{\infty}x^{(k^2+k)/2}.
\end{split}
\end{displaymath}
In both pairs of power series one sees that the non-zero
coefficients are supported on the same terms.
For $t=3$, we then have
\begin{equation}
\begin{split}
F_9(x)&=\sum_{n=0}^{\infty} a(n)x^n:=\prod_{n=1}^{\infty}
(1-x^n)^8\\
&=1-8x+20x^2-70x^4+\cdots-520x^{14}+57x^{16}+560x^{17}+182x^{20}+\cdots
\end{split}\end{equation}
and
\begin{equation}
\begin{split}
C_3(x)&=\sum_{n=0}^{\infty} b(n)x^n:=\prod_{n=1}^{\infty}
\frac{(1-x^{3n})^3}{1-x^n}\\
&=1+x+2x^2+2x^4+\cdots+2x^{14}+3x^{16}+2x^{17}+2x^{20}+\cdots.
\end{split}
\end{equation}

\begin{remark} It is clear that $b(n)=c_3(n)$.
\end{remark}

In accordance with the elementary observations when $t=1$ and $2$,
one notices that the non-zero coefficients of $F_9(x)$ and $C_3(x)$ appear to be
supported on the same terms.
Based on substantial numerical evidence, the first author made the following
conjecture.

\smallskip
\noindent
{\bf Conjecture} 4.6. (Conjecture 4.6 of \cite{Han2}) \ \ \newline
{\it Assuming the notation above, we have that $a(n)=0$ if and only if
$b(n)=0$.}
\smallskip

\begin{remark}
The obvious generalization of Conjecture 4.6 and the examples above
is not true for $t=4$. In particular, one easily finds that
\begin{displaymath}
\begin{split}
F_{16}(x)&=1-15x+90x^2-\cdots+641445x^{52}+1537330x^{54}+\cdots,\\
C_4(x)&=1+x+2x^2+3x^3+\cdots+5x^{52}+8x^{53}+10x^{54}+\cdots.
\end{split}
\end{displaymath}
The coefficient of $x^{53}$ vanishes in $F_{16}(x)$ and is non-zero
in $C_4(x)$.
\end{remark}

Here we prove that Conjecture 4.6 is true. We have the following theorem.

\begin{theorem}\label{hanconjecture}
Assuming the notation above, we have that $a(n)=0$ if and
only if $b(n)=0$. Moreover, we have that $a(n)=b(n)=0$
precisely for those non-negative $n$ for which
$\ord_p(3n+1)$ is odd for some prime $p\equiv 2\pmod 3$.
\end{theorem}
\begin{remark} As usual, $\ord_p(N)$ denotes the power of a prime $p$
dividing an integer $N$.
\end{remark}

\begin{remark}
Theorem~\ref{hanconjecture} shows that $a(n)=b(n)=0$ in a systematic
way. The vanishing coefficients are associated to primes $p\equiv
2\pmod 3$. If $n\equiv 1\pmod 3$ has the property that $\ord_p(n)$
is odd, then we have
\begin{displaymath}
a\left(\frac{n-1}{3}\right)=b\left(\frac{n-1}{3}\right)=0.
\end{displaymath}
For example, since $\ord_5(10)=1\equiv 1\pmod 2$, we have that
$a(3)=b(3)=0$.
\end{remark}

As an immediate corollary, we have the following.

\begin{corollary}\label{EZ}
For positive integers $N$, we have that
\begin{displaymath}
\sum_{\lambda \vdash N}
\prod_{h\in \mathcal{H}(\lambda))} \left(1-\frac{9}{h^2}\right)=0
\end{displaymath}
if and only if there are  no $3$-core partitions of $N$.
\end{corollary}

Theorem~\ref{hanconjecture} implies that ``almost all'' of the
$a(n)$ and $b(n)$ are 0. More precisely, we have the following.

\begin{corollary}\label{lacunarity}
Assuming the notation above, we have that
\begin{displaymath}
\lim_{X\rightarrow +\infty} \frac{\# \{ 0\leq n\leq X \ : \
a(n)=b(n)=0\}}{X}=1.
\end{displaymath}
\end{corollary}

\section*{Acknowledgements}
The authors thank Mihai Cipu for insightful comments 
related to Conjecture 4.6.

\section{Proofs}
It is convenient to renormalize the functions $a(n)$ and $b(n)$
using the series
\begin{equation}\label{Aq}
\begin{split}
\mathcal{A}(z)&=\sum_{n=1}^{\infty}a^*(n)q^n:=\sum_{n=0}^{\infty}a(n)q^{3n+1}\\
    &=q-8q^4+20q^7-70q^{13}+64q^{16}+56q^{19}-125q^{25}-160q^{28}+\cdots.
\end{split}
\end{equation}
and
\begin{equation}\label{Bq}
\begin{split}
\mathcal{B}(z)&=\sum_{n=1}^{\infty}b^*(n)q^n:=\sum_{n=0}^{\infty}b(n)q^{3n+1}\\
&=q+q^4+2q^7+2q^{13}+q^{16}+2q^{19}+q^{25}+2q^{28}+\cdots.
\end{split}
\end{equation}
Here we have that $z\in \H$, the upper-half of the complex plane,
and we let $q:=e^{2\pi i z}$. We make these changes since
$\mathcal{A}(z)$ and $\mathcal{B}(z)$ are examples of two special
types of modular forms (for background on modular forms, see
\cite{Bump, Iwaniec, Miyake, OnoCBMS}). The modularity of these two
series follows easily from the properties of Dedekind's eta-function
\begin{equation}
\eta(z):=q^{\frac{1}{24}}\prod_{n=1}^{\infty}(1-q^n).
\end{equation}
The proofs of Theorem~\ref{hanconjecture} and
Corollary~\ref{lacunarity}  shall rely on exact formulas we derive
for the numbers $a^*(n)$ and $b^*(n)$.

\subsection{Exact formulas for $a^*(n)$}

The modular form $\mathcal{A}(z)$ given by
\begin{displaymath}
\mathcal{A}(z)=\eta(3z)^8=\sum_{n=1}^{\infty}a^*(n)q^n
\end{displaymath}
is in $S_4(\Gamma_0(9))$, the space of weight 4 cusp forms on
$\Gamma_0(9)$. This space is one dimensional (see Section 1.2.3 in
\cite{OnoCBMS}). Therefore, every cusp form in the space is a
multiple of $\mathcal{A}(z)$. It turns out that $\mathcal{A}(z)$ is
a form with {\it complex multiplication}.

We now briefly recall the notion of a newform with complex
multiplication (for example, see Chapter 12 of \cite{Iwaniec} or
Section 1.2 of \cite{OnoCBMS}, \cite{Ribet}). Let $D<0$ be the
fundamental discriminant of an imaginary quadratic field
$K=\Q(\sqrt{D})$. Let $O_K$ be the ring of integers of $K$, and let
$\chi_K:=\leg{D}{\bullet}$ be the usual Kronecker character
associated to $K$. Let $k\geq 2$, and let $c$ be a Hecke character
of $K$ with exponent $k-1$ and conductor $\mathfrak{f}_c$, a
non-zero ideal of $O_K$. By definition, this means that
\begin{displaymath}
  c: I(\mathfrak{f}_c)\longrightarrow \C^{\times}
\end{displaymath}
is a homomorphism, where $I(\mathfrak{f}_c)$ denotes the group of
fractional ideals of $K$ prime to $\mathfrak{f}_c$. In particular,
this means that
\begin{displaymath}
  c(\alpha O_K)=\alpha^{k-1}
\end{displaymath}
for $\alpha\in K^{\times}$ for which $\alpha\equiv 1\ {\text {\rm
mod}}^{\times}
 \mathfrak{f}_c$.
To $c$ we naturally associate a Dirichlet character $\omega_c$
defined, for every integer $n$ coprime to $\mathfrak{f}_c$, by
\begin{displaymath}
   \omega_c(n):=\frac{c(nO_K)}{n^{k-1}}.
\end{displaymath}
Given this data, we let
\begin{equation}\label{CMform}
     \Phi_{K,c}(z):=\sum_{\mathfrak{a}} c(\mathfrak{a})q^{N(a)},
\end{equation}
where $\mathfrak{a}$ varies over the ideals of $O_K$ prime to
$\mathfrak{f}_c$, and where $N(\mathfrak{a})$ is the usual ideal
norm. It is known that $\Phi_{K,c}(z)\in S_k(|D|\cdot
N(\mathfrak{f}_c),\chi_K\cdot \omega_c)$ is a normalized newform.

Using this theory, we obtain the following theorem.

\begin{theorem}\label{astar}
Assume the notation above. Then the following are true:
\begin{enumerate}
\item If $p=3$ or $p\equiv 2\pmod 3$ is prime, then $a^*(p)=0$.
\item If $p\equiv 1\pmod 3$ is prime, then
$$
a^*(p)=2x^3-18xy^2,
$$
where $x$ and $y$ are integers for which $p=x^2+3y^2$ with $x\equiv
1\pmod 3$.
\end{enumerate}
\end{theorem}
\begin{remark} It is a classical fact that every prime $p\equiv
1\pmod 3$ is of the form $x^2+3y^2$. Moreover, there is a unique
pair of positive integers $x$ and $y$ for which $x^2+3y^2=p$.
Therefore, the formula for $a^*(p)$ is well defined.
\end{remark}
\begin{proof} There is a form with complex multiplication in
$S_4(\Gamma_0(9))$. Following the recipe above, it is obtained by
letting $k=4$, $\Q(\sqrt{D})=\Q(\sqrt{-3})$ and $\mathfrak{f}_c:=
(\sqrt{-3})$. For primes $p$, the coefficients of $q^p$ in this form
agree with the claimed formulas. Since $S_4(\Gamma_0(9))$ is one
dimensional, this form must be $\mathcal{A}(z)$.
\end{proof}

Using this theorem, we obtain the following immediate corollary.

\begin{corollary}\label{aprimepower}
The following are true about $a^*(n)$.
\begin{enumerate}
\item If $m$ and $n$ are coprime positive integers, then
$$
a^*(mn)=a^*(m)a^*(n).
$$
\item For every positive integer $s$, we have that $a^*(3^s)=0$.
\item If $p\equiv 2\pmod 3$ is prime and $s$ is a positive integer,
then
$$
a^*(p^s)=\begin{cases} 0 \ \ \ \ \ &{\text {\rm if}}\ s\ {\text {\rm
is odd,}}\\
(-1)^{s/2}p^{3s/2} \ \ \ \ \ &{\text {\rm if}}\ s\ {\text {\rm is
even.}}
\end{cases}
$$
\item If $p\equiv 1\pmod 3$ is prime and $s$ is a positive integer,
then $a^*(p^s)\neq 0$. Moreover, we have that
$$
a^*(p^s)\equiv (8x^3)^s\pmod p,
$$
where $p=x^2+3y^2$ with $x\equiv 1\pmod 3$.
\end{enumerate}
\end{corollary}
\begin{proof}
Since $S_4(\Gamma_0(9))$ is one dimensional and since $a^*(1)=1$, it
follows that $\mathcal{A}(z)$ is a normalized Hecke eigenform. Claim
(1) is well known to hold for all normalized Hecke eigenforms.

Claim (2) follows by inspection since $a^*(n)=0$ if $n\equiv
0,2\pmod 3$.

To prove claims (3) and (4), we note that since $\mathcal{A}(z)$ is
a normalized Hecke eigenform on $\Gamma_0(9)$, it follows, for every
prime $p\neq 3$, that
\begin{equation}\label{powerrecursion}
a^*(p^s)=a^*(p)a^*(p^{s-1})-p^3a^{*}(p^{s-2}).
\end{equation}
If $p\equiv 2\pmod 3$ is prime, then Theorem~\ref{astar} implies
that
$$
a^*(p^s)=-p^3a(p^{s-2}).
$$
Claim (3) now follows by induction since $a^*(1)=1$ and $a^*(p)=0$.

Suppose that $p\equiv 1\pmod 3$ is prime. By Theorem~\ref{astar}, we
know that $a^*(p)\neq 0$. More importantly, we have that
\begin{displaymath}
a^*(p)\equiv 8x^3\pmod p,
\end{displaymath}
where $p=x^2+3y^2$ with $x\equiv 1\pmod 3$. To see this, one merely
observes that
$$
2x^3-18xy^2=2x(x^2-9y^2)=2x(x^2-3(p-x^2))\equiv 8x^3\pmod p.
$$
Since $|x|\leq \sqrt{p}$ and is non-zero, it follows that
$a^*(p)\equiv 8x^3\not \equiv 0\pmod p$. By (\ref{powerrecursion}),
we then have that
$$
a^*(p^s)\equiv a^*(p)a^*(p^{s-1})\equiv 8x^3a^*(p^{s-1})\pmod p.
$$
By induction, it follows that $a^*(p^s)\equiv (8x^3)^s\pmod p$,
which is non-zero modulo $p$. This proves claim (4).
\end{proof}

\begin{example} Here we give some numerical examples of the formulas
for $a^*(n)$.

\noindent 1) One easily finds that $a^*(13)=-70$. The prime $p=13$
is of the form $x^2+3y^2$ where $x=1$ and $y=2$. Obviously,
$x=1\equiv 1\pmod 3$, and so Theorem~\ref{astar} asserts that
$a^*(13)=2\cdot 1^3-18\cdot 1\cdot 2^2=-70$.

\smallskip
\noindent 2) We have that $a^*(13)=-70$ and $a^*(16)=64$. One easily
checks that $a^*(13\cdot 16)=a^*(208)=-70\cdot 64=-4408$. This is an
example of Corollary~\ref{aprimepower} (1).

\smallskip
\noindent 3) If $p=5$ and $s=3$, then Corollary~\ref{aprimepower}
(3) asserts that $a^*(5^3)=0$. If $p=5$ and $s=4$, then it asserts
that $a^*(5^4)=5^6=15625$. One easily checks both evaluations
numerically.

\smallskip
\noindent 4) Now we consider the prime $p=13\equiv 1\pmod 3$. Since
$x=1$ and $y=2$ for $p=13$, Corollary~\ref{aprimepower} (4) asserts
that $a^*(13^s)\equiv 8^s\pmod {13}$. One easily checks that
\begin{displaymath}
\begin{split}
a^*(13)&=-70\equiv 8\pmod{13},\\
a^*(13^2)&=2703\equiv 8^2\pmod{13},\\
a^*(13^3)&=-35420\equiv 8^3\pmod{13}.
\end{split}
\end{displaymath}

\end{example}

\subsection{Proof of Theorem~\ref{hanconjecture} and
Corollary~\ref{lacunarity}}

Before we prove Theorem~\ref{hanconjecture}, we recall a known
formula for $b(n)$ (also see Section 3 of \cite{GranvilleOno}), the
number of 3-core partitions of $n$.

\begin{lemma}\label{3core}
Assuming the notation above, we have that
$$
\mathcal{B}(z)=\sum_{n=1}^{\infty}b^*(n)q^n=
\sum_{n=0}^{\infty}b(n)q^{3n+1}=\sum_{n=0}^{\infty} \sum_{d\mid
3n+1} \leg{d}{3}q^{3n+1},
$$
where $\leg{\bullet}{3}$ denotes the usual Legendre symbol modulo 3.
\end{lemma}
\begin{proof} We have that
$\mathcal{B}(z)=\eta(9z)^3/\eta(3z)$ is in $M_1(\Gamma_0(9),\chi)$,
where $\chi:=\leg{-3}{\bullet}$. The lemma follows easily from this
fact. One may implement the theory of weight 1 Eisenstein series to
obtain the desired formulas.

Alternatively, one may use the weight 1 form
$$
\Theta(z)=\sum_{n=0}^{\infty}c(n)q^n:=\sum_{x,y\in
\Z}q^{x^2+xy+y^2}=1+6q+6q^3+6q^4+12q^7+6q^9+\cdots.
$$
Using the theory of twists, we find that
\begin{displaymath}
\begin{split}
\widetilde{\Theta}(z)=\sum_{n\equiv 1\pmod 3}c(n)q^n&=
6q+6q^4+12q^7+12q^{13}+6q^{16}+12q^{19}+6q^{25}+\cdots\\
&=6\left(q+q^4+2q^7+2q^{13}+q^{16}+2q^{19}+q^{25}+\cdots\right).
\end{split}
\end{displaymath}
By dimensionality (see Section 1.2.3 of \cite{OnoCBMS}) we have that
$\mathcal{B}(z)=6\widetilde{\Theta}(z)$. The claimed formulas for
the coefficients follows easily from the fact that $x^2+xy+y^2$
corresponds to the norm form on the ring of integers of
$\Q(\sqrt{-3})$.

\end{proof}

\begin{example} The only divisors of primes $p\equiv 1\pmod 3$
are 1 and $p$, and so we have that
$b^*(p)=1+\leg{p}{3}=1+\leg{1}{3}=2$.
\end{example}

\begin{proof}[Proof of Theorem~\ref{hanconjecture}]
The theorem follows immediately from Theorems~\ref{astar},
\ref{aprimepower} and Lemma~\ref{3core}. One sees that the only
$n\equiv 1\pmod 3$ for which $a^*(n)=0$ are those $n$ for which
$\ord_p(n)$ is odd for some prime $p\equiv 2\pmod 3$. The same
conclusion holds for $b^*(n)$. Using the fact that
$$
a(n)=a^*(3n+1) \ \ \ {\text {\rm and}}\ \ \ b(n)=b^*(3n+1),
$$
the theorem follows.
\end{proof}

\begin{proof}[Proof of Corollary~\ref{lacunarity}]
In a famous paper \cite{Serrelacunarity}, Serre proved that ``almost
all'' of the coefficients of a modular form with complex
multiplication are zero. This implies that almost all of the
$a^*(n)$ are zero. The result now follows thanks to
Theorem~\ref{hanconjecture}.
\end{proof}

\end{document}